\documentclass[a4paper]{amsart}
\usepackage{amsmath,amssymb}
\usepackage{amsmath}
\usepackage{comment}
\usepackage{amsthm}
\usepackage{mathtools}
\usepackage{amssymb}
\usepackage{framed}
\usepackage[hidelinks]{hyperref}
\usepackage{xcolor}
\usepackage{graphicx}
\usepackage[a4paper,  margin=2.5cm]{geometry}

\newtheorem{theo}{Theorem}[section]
\newtheorem{pr}{Proposition}[section]
\newtheorem{lm}[pr]{Lemma}

\newtheorem{re}[pr]{Remark}
\newcommand{\er}{\mathbf e_r}
\newcommand{\ep}{\mathbf e_\phi}
\newcommand{\ez}{\mathbf e_z}

\DeclareMathOperator{\di}{div}
\DeclareMathOperator{\cu}{curl}

\begin{document}

\title[Seifert type MHS equilibria]{Seifert type magnetohydrostatic equilibria in rotationally symmetric toroidal domains}

\author{Daniel Peralta-Salas}
\address{Instituto de Ciencias Matem\'aticas, Consejo Superior de
  Investigaciones Cient\'\i ficas, 28049 Madrid, Spain}
\email{dperalta@icmat.es}

\author{Radu Slobodeanu}
\address{Faculty of Physics, University of Bucharest, P.O. Box Mg-11, RO--077125 Bucharest-M\u agurele, Romania}
\email{radualexandru.slobodeanu@g.unibuc.ro}


\begin{abstract}
We construct smooth magnetohydrostatic equilibria in rotationally symmetric toroidal domains whose magnetic lines define an $(m,n)$ Seifert fibration for arbitrary coprime integers $m,n$.
\end{abstract}

\maketitle

\section{Introduction}\label{S:intro}

In magnetohydrodynamics (MHD) \cite{KK}, plasma equilibria are described by a magnetic vector field $B$ in a smooth bounded domain $\Omega \subset \mathbb{R}^3$ satisfying
\begin{equation}\label{MHSeq}
\left\{
\begin{array}{lll}
B\times\cu  B +\nabla P=0  & \text{ in } \Omega\\[2mm]
\di B = 0 & \text{ in } \Omega \\[2mm]
B\cdot  N =0, \quad P\equiv const. & \text{ on } \partial \Omega
\end{array} \right.
\end{equation}
for some (plasma) \emph{pressure} function $P$. Here $N$ is the outward unit normal vector on $\partial\Omega$. A vector field $B$ satisfying Equations~\eqref{MHSeq} will be called an MHS equilibrium.

In this article we are concerned with axisymmetric MHS equilibria. In this case, it is well known that Equations~\eqref{MHSeq} reduce to a scalar elliptic PDE for a rotationally symmetric stream function $\psi$. Indeed, using cylindrical coordinates $(r,\phi,z)$, with $\phi\in\mathbb R/(2\pi\mathbb Z)$, and the associated orthonormal frame
$(\mathbf e_r,\mathbf e_\phi,\mathbf e_z)$, the MHS equilibrium reads as
\begin{equation*}
B=-\frac{\partial_z\psi}{r}\mathbf e_r+\frac{\partial_r\psi}{r}\mathbf e_z+\frac{F(\psi)}{r}\mathbf e_\phi\,,
\end{equation*}
where the stream function $\psi(r,z)$ satisfies the Grad-Shafranov equation
\begin{equation*}
\partial_{rr}\psi +\partial_{zz}\psi -\frac1r\partial_r\psi=r^2H'(\psi)-\frac12(F^2)'(\psi)\,,
\end{equation*}
and the functions $F$ and $H$ can be picked freely. We observe that the function $H$ is, in fact, minus the plasma pressure $P$.

Our goal in this paper is to construct MHS equilibria in rotationally symmetric toroidal domains
whose trajectories define a Seifert fibration with invariants $(m,n)$ \cite{Seif}, for arbitrary nonzero coprime integers $m, n$. The magnetic axis is the only exceptional fibre. MHS equilibria all whose orbits are closed of type $(0,1)$ were constructed by Lortz~\cite{Lortz} in non-axisymmetric toroidal domains. There are also many examples of axisymmetric equilibria all whose orbits are closed and poloidal (i.e., of type $(1,0)$), see~\cite{Salat}. The case $(m,n)=(2,1)$ can be explicitly obtained using the Solov'ev solution~\cite{solo}, and it has been recently used in~\cite{landre} in order to construct counterexamples to Grad's conjecture. We observe that explicit and analytic MHS equilibria defining arbitrary Seifert fibrations were constructed on $\mathbb{S}^3$ (endowed with the canonical metric)~\cite{Ps,slo20}, but the Euclidean proof is much more complicated. Sufficient conditions for an MHS equilibrium to have all orbits closed can be seen in~\cite{Perre}.

In the statement of our main theorem we use the following notation. Given a rotationally symmetric toroidal domain $\Omega$, we denote its meridian $\mu$ as the intersection $\partial\Omega\cap \{x=0\}$ and
its longitude $\lambda$ as a circle $\{r=r_0,z=z_0\}$ in $\Omega$. The curve $\lambda$ is oriented by increasing $\phi$ and $\mu$ is oriented counterclockwise (when seen from the $\{x>0\}$ half-space).

\begin{theo}\label{prop:sf-existence}
There is $\varepsilon_0>0$ such that, for every
$0<\varepsilon<\varepsilon_0$ and every pair of coprime nonzero
integers $m,n$, there is an axisymmetric MHS equilibrium
$B\in C^\infty(\overline{\Omega_\varepsilon})$ in a smooth rotationally symmetric toroidal domain $\Omega_\varepsilon$ satisfying Equations~\eqref{MHSeq} and the following properties:
\begin{enumerate}
\item  The field $B$ is nowhere zero, and the level sets of the pressure $P\in C^\infty(\overline{\Omega_\varepsilon})$ define a foliation of nested toroidal surfaces enclosing a critical line $C$.
\item The curve $C$, which is a magnetic line, is a circle of the form $\{r=1+\varepsilon^2+O(\varepsilon^3),z=O(\varepsilon^3)\}$, and
every other orbit of $B$ is closed and lies on a (toroidal) level set of $P$,
with winding class $m\mu+n\lambda$ (where $\mu$ and $\lambda$ are defined and oriented as above). In particular,
\begin{equation}\label{eq:sf-linking}
 \operatorname{lk}(\Gamma,C)=-m
\end{equation}
for every noncore orbit $\Gamma$.
\end{enumerate}
Therefore, the orbit foliation defines a Seifert fibration whose only exceptional fibre is $C$, and the regular fibres make $m$ turns in the poloidal direction and $n$ turns in the toroidal direction.
\end{theo}

\begin{re}
Taking $(m,n)=(N,1)$ gives arbitrary core linking $N$ without an
exceptional fibre. Taking $(m,n)=(3,2)$ gives noncore trajectories of
trefoil type and a core of multiplicity two.
\end{re}

The domain $\Omega_\varepsilon$ in Theorem~\ref{prop:sf-existence} is a neighborhood of the circle $\{r=1,z=0\}$ of thickness $O(\varepsilon)$, while the MHS equilibrium (and its derivatives) is bounded by $O(\varepsilon^2)$. We want to emphasize that the smallness threshold $\varepsilon_0$, the domain $\Omega_\varepsilon$ and the pressure $P$ do not depend on $m$, $n$.

The idea of the proof consists in constructing solutions to the Grad-Shafranov equation that in suitable local coordinates $(X,Y)$ are diffeomorphic to $\frac12(X^2+Y^2)$, and the diffeomorphism is area-preserving. This is obtained using the Nash--Moser theorem for smooth tame right inverses (see Section~\ref{S:elliptic}). With the help of this diffeomorphism and a suitable nonconstant scaling, $\Omega_\varepsilon$ transforms onto $D\times \mathbb S^1$ (with $D$ the unit disk) and $B$ reads as
$$
(-Y\partial_X+X\partial_Y)+\kappa\partial_\phi\,,
$$
with $\kappa:=n/m$. From this expression, the claimed Seifert structure for the magnetic field $B$ becomes clear. In fact, the construction works for any real number $\kappa$, so taking $\kappa\not\in\mathbb Q$, we obtain MHS equilibria whose orbits are ergodic on each toroidal level set of $P$, except for the core line $C$, which is a periodic orbit of $B$. Taking $\kappa=0$ we obtain MHS equilibria all whose orbits are closed and poloidal. See Section~\ref{S:proof} for a detailed proof of Theorem~\ref{prop:sf-existence}.


\section{Solutions to a semilinear elliptic PDE using area-preserving diffeomorphisms}\label{S:elliptic}

\noindent In this section we work on the unit disk, which we parametrize by
$$D:=\{q=(x,y):x^2+y^2<1\}\,,$$
and denote the symplectic matrix as $J$ and the angular derivative as $\mathcal R$, i.e.,
\[
 J:=\begin{pmatrix}0&-1\\1&0\end{pmatrix},\qquad
 \mathcal R:=-y\partial_x+x\partial_y.
\]
When acting on vector fields, the operator $\mathcal R$ acts componentwise. Write
$$R_t:=\begin{pmatrix}
\cos t & -\sin t\\ \sin t & \cos t
\end{pmatrix}$$
as the counterclockwise rotation, and define the averaging operators
\begin{equation}\label{eq:sf-projections}
 (\mathcal{P}a)(q):=\frac1{2\pi}\int_0^{2\pi}a(R_tq)\,dt,
 \qquad \mathcal Q:=I-\mathcal P\,,
\end{equation}
for $a\in C^\infty(\overline D)$. We use the Sobolev norms $H^s(D)$, with $s$ a nonnegative integer, to grade the spaces of smooth functions. The operators $\mathcal P$ and $\mathcal Q$ extend to bounded operators on each $H^s(D)$.

For every function $g\in C^\infty(\overline D)$ that satisfies $\mathcal Pg=0$, let us set its zero-mean angular inverse
\begin{equation}\label{eq:sf-angular-inverse}
 (\mathcal Sg)(q):=\frac1{2\pi}\int_0^{2\pi}t\,g(R_tq)\,dt.
\end{equation}
In this case, since $\partial_t g(R_tq)=\mathcal Rg(R_tq)$, integration by parts yields
\[
\mathcal R\mathcal Sg=g\,,
\]
(this justifies that we are talking about an ``angular inverse'') and, moreover, $\mathcal P\mathcal Sg=0$ because $\mathcal P$ and $\mathcal S$ commute.
It is also clear from the integral formula that $\mathcal S$ is bounded on $H^s(D)\cap\ker\mathcal P$ for every nonnegative integer $s$.

For a smooth vector field $v=(v_1,v_2)$, we write
\[
 \operatorname{curl}_2v:=\partial_xv_2-\partial_yv_1,
\]
and denote by $\nu(q)=q$, for $q\in\partial D$, the outward unit normal to the disk. The normalization used below is the explicit condition
$\mathcal P\operatorname{curl}_2v=0$.

Next we introduce the following auxiliary elliptic operator, which is an $\varepsilon$-perturbation of the Laplacian $\Delta=\partial_{xx}+\partial_{yy}$:
\begin{equation}\label{eq:sf-operator}
 L_\varepsilon
 :=e^{-4\varepsilon x}(\partial_{xx}-2\varepsilon\partial_x)
       +e^{-2\varepsilon x}\partial_{yy}.
\end{equation}
As we will see in the next section, this operator is essentially the Grad-Shafranov operator after an explicit $\varepsilon$-dependent change of variables.

The following proposition shows that for any $\varepsilon>0$ small enough, the radial function
$$
\Psi_0=\frac12(x^2+y^2)
$$
can be deformed with an area-preserving diffeomorphism of the disk so that the new function $\Psi_\varepsilon$ satisfies the elliptic equation
$$
L_\varepsilon\Psi_\varepsilon
                 =f_\varepsilon(\Psi_\varepsilon)
$$
for some real-valued function $f_\varepsilon$, in the spirit of the ideas introduced by
Constantin, Drivas, and Ginsberg~\cite[Theorem 3.1]{CDG2021}.

\begin{pr}\label{lem:sf-deformation}
For all sufficiently small $\varepsilon>0$, there exist a diffeomorphism
$\gamma_\varepsilon:\overline D\to\overline D$ of class
$C^\infty$ and a function
$f_\varepsilon\in C^\infty([0,1/2])$ such that
\begin{equation}\label{eq:sf-deformation}
 \det D\gamma_\varepsilon=1,\qquad
 L_\varepsilon\Psi_\varepsilon
                 =f_\varepsilon(\Psi_\varepsilon),
\end{equation}
with $\Psi_\varepsilon:=\Psi_0\circ\gamma_\varepsilon^{-1}\in C^\infty(\overline D)$.
The map $\gamma_\varepsilon$ depends smoothly on $\varepsilon$, extends to $\gamma_0=\mathrm{id}$, is close to the identity in the $C^\infty$ topology, and can be chosen with $\mathcal P\operatorname{curl}_2(\gamma_\varepsilon-\mathrm{id})=0$. Moreover,
$\Psi_\varepsilon$ has a unique nondegenerate minimum at
$q_\varepsilon=\gamma_\varepsilon(0)$. Furthermore,
\begin{equation}\label{eq:sf-hamiltonian-conjugacy}
 (\gamma_\varepsilon)_*X_{\Psi_0}=X_{\Psi_\varepsilon}\,,
\end{equation}
where $X_\Psi:=J\nabla\Psi$, so every nonconstant orbit of $X_{\Psi_\varepsilon}$ has minimal
period $2\pi$.
\end{pr}

\begin{proof}
For a smooth vector field $v\in C^\infty(\overline D;\mathbb R^2)$ such that $q\mapsto q+v(q)$ is a diffeomorphism onto its image, set
\[
 \gamma(q):=q+v(q),\qquad A:=D\gamma,\qquad
 W:=-J\mathcal R\gamma,\qquad M:=(DW)A^{-1}\,.
\]
(We use the Jacobian convention $(D\gamma)_{ij}=\partial_j\gamma_i$.)
Consider the function
\begin{equation}\label{eq:sf-residual}
 E_\varepsilon(v)
 :=e^{-4\varepsilon\gamma_1}(M_{11}-2\varepsilon W_1)
             +e^{-2\varepsilon\gamma_1}M_{22}\,,
\end{equation}
where the subscripts in $M$ and $W$ indicate the corresponding matrix components. The function $E_\varepsilon$ is related to the operator $L_\varepsilon$ as follows.
Assume that $\det A=1$ and set $\Psi:=\Psi_0\circ\gamma^{-1}$, then
\[
 W=-J\mathcal R\gamma=-JAJq=A^{-T}q=(\nabla\Psi)\circ\gamma,
\]
where we have used that $A^{-T}=-JAJ$ because $\det A=1$, and $(D\gamma)^T \nabla\Psi\circ \gamma = q$ (take the derivative in the definition of $\Psi$). From this identity and the definition of $M$, it is straightforward that
$$M=(D^2\Psi)\circ\gamma\,.$$
Plugging these equations into the definition of $E_\varepsilon(v)$ we obtain
$$E_\varepsilon(v)=(L_\varepsilon\Psi)\circ\gamma$$
whenever $\det A=1$, which establishes the aforementioned relation between $E_\varepsilon$ and $L_\varepsilon$. No normalization $\mathcal P\operatorname{curl}_2v=0$ is imposed in the definition of $E_\varepsilon(v)$.

Our goal is to solve the equation 
$$\mathcal Q E_\varepsilon(\gamma-\mathrm{id})=0$$ 
directly on the space of smooth area-preserving diffeomorphisms of the disk onto itself. The area preservation and the boundary conditions are then automatic. To obtain a smooth family on a single interval of $\varepsilon$, we use the Nash--Moser theorem for smooth tame right inverses; the normalization will be imposed afterwards without changing the stream function.

Let $\mathcal D_{\mathrm{vol}}$ denote the smooth tame Fr\'echet manifold of $C^\infty$ area-preserving diffeomorphisms of $\overline D$ onto itself, and define
\[
 \mathcal G_\infty:=\{g\in C^\infty(\overline D):\mathcal Pg=0\}.
\]
We use the usual Sobolev gradings on these spaces. The manifold structure of $\mathcal D_{\mathrm{vol}}$ is the standard one for volume-preserving diffeomorphisms of a compact manifold with boundary~\cite{EM1970,Hamilton1982}. The tangent vector space of $\mathcal D_{\mathrm{vol}}$ at the identity consists of the divergence-free vector fields tangent to $\partial D$, which on the disk are precisely
\[
 J\nabla\beta,\qquad \beta\in C^\infty(\overline D),\qquad
 \beta|_{\partial D}=0\,.
\]
We recall that the space $\mathcal G_\infty$ is tame because $\mathcal Q$ is a tame projection on $C^\infty(\overline D)$.

For $\gamma\in\mathcal D_{\mathrm{vol}}$ close to the identity, write
\[
 \Psi_\gamma:=\Psi_0\circ\gamma^{-1},\qquad
 E_{\gamma,\varepsilon}:=(L_\varepsilon\Psi_\gamma)\circ\gamma,
 \qquad
 \Phi(\gamma,\varepsilon):=\mathcal Q E_{\gamma,\varepsilon}.
\]
Since $\det\,D\gamma=1$, it is clear that $E_{\gamma,\varepsilon}=E_\varepsilon(\gamma-\mathrm{id})$ and $\Phi(\mathrm{id},0)=0$. Obviously, the image of $\Phi(\gamma,\varepsilon)$ is contained in the space $\mathcal G_\infty$. Our next goal is to construct a smooth tame family of right inverses of $D_\gamma\Phi$ on a fixed neighborhood of $(\mathrm{id},0)$.

In order to compute the differential $D_\gamma\Phi$, we follow the usual approach of constructing a curve in $\mathcal D_{\mathrm{vol}}$ starting at $\gamma$. To this end,
given $\beta\in C^\infty(\overline D)$ with $\beta|_{\partial D}=0$, let $\phi_t$ be the flow of $J\nabla\beta$ and set $\gamma_t:=\gamma\circ\phi_t$. Obviously, $\gamma_t\in \mathcal D_{\mathrm{vol}}$ for all $t\in\mathbb R$ and $\gamma_0=\gamma$. Then
\[
 \dot\gamma_t=D\gamma_t\,J\nabla\beta,\qquad
 \dot\Psi_{\gamma_t}\circ\gamma_t
 =-J\nabla\beta\cdot\nabla\Psi_0=\mathcal R\beta\,,
\]
where $\cdot$ denotes $t$-derivative. Introduce the conjugated elliptic operator
\[
 \mathcal A_{\gamma,\varepsilon}u
 :=\bigl[L_\varepsilon(u\circ\gamma^{-1})\bigr]\circ\gamma.
\]
In view of the previous identities, we can compute $D_\gamma \Phi$ as
\[
D_\gamma\Phi(\gamma,\varepsilon)
       [D\gamma\,J\nabla\beta]
= \frac{d}{dt}\Big|_{t=0}\Phi(\gamma_t,\varepsilon)=\mathcal Q \Big[\frac{d}{dt}\Big|_{t=0}E_{\gamma_t,\varepsilon}\Big]\,,
\]
where we have used that the linear operator $\mathcal Q$ commutes with the $t$-derivative. A straightforward computation then yields
\begin{equation}\label{eq:sf-smooth-linearization}
 D_\gamma\Phi(\gamma,\varepsilon)
       [D\gamma\,J\nabla\beta]
 =\mathcal Q\bigl[\mathcal A_{\gamma,\varepsilon}(\mathcal R\beta)
                       +J\nabla\beta\cdot\nabla E_{\gamma,\varepsilon}\bigr].
\end{equation}
To compute the coefficients of this operator, we use that
\[
 L_\varepsilon u
 =e^{-2\varepsilon x}\operatorname{div}
       \left(\begin{pmatrix}e^{-2\varepsilon x}&0\\0&1\end{pmatrix}
                                      \nabla u\right)\,.
\]
Since $\det D\gamma=1$, using the identity 
\[
(\operatorname{div} X)\circ\gamma=\operatorname{div}\big(A^{-1}(X\circ\gamma)\big)\,,
\]
which holds for any smooth vector field $X$, after a few computations we obtain
\begin{equation}\label{eq:sf-smooth-coefficients}
 \begin{split}
 \mathcal A_{\gamma,\varepsilon}u
   &=\operatorname{div}(G\nabla u)+d\cdot\nabla u,\\
 G&:=A^{-1}\begin{pmatrix}e^{-4\varepsilon\gamma_1}&0\\
                         0&e^{-2\varepsilon\gamma_1}\end{pmatrix}A^{-T},
 \qquad d:=2\varepsilon G\nabla\gamma_1\,.
 \end{split}
\end{equation}
Recall that we introduced the notation $A=D\gamma$. In particular, $G$ is symmetric and positive definite. Moreover,
\begin{equation}\label{eq:sf-smooth-E}
 E_{\gamma,\varepsilon}
 =\operatorname{tr}G+(\operatorname{div}G)\cdot q+d\cdot q\,,
\end{equation}
where $(\operatorname{div}G)_j:=\sum_i\partial_iG_{ij}$. Equation~\eqref{eq:sf-smooth-linearization} therefore becomes
\begin{equation}\label{eq:sf-smooth-tangential}
 D_\gamma\Phi(\gamma,\varepsilon)
       [D\gamma\,J\nabla\beta]=\mathcal Q\bigl[\operatorname{div}(G\nabla\mathcal R\beta)
        +d\cdot\nabla\mathcal R\beta+c\cdot\nabla\beta\bigr]\,,
\end{equation}
with 
$$c:=-J\nabla E_{\gamma,\varepsilon}\,.$$ 
At $(\gamma,\varepsilon)=(\mathrm{id},0)$, these coefficients satisfy $G=I$ and $d=c=0$.

In what follows we normalize the function $\beta$ by $\mathcal P\beta=0$ and write $w:=\mathcal R\beta$. Since $\mathcal R$ and $\mathcal S$ commute because $\mathcal P\beta=0$, we infer that $\beta=\mathcal Sw$. Moreover, $\mathcal P w=0$. With this notation, we obtain from Equation~\eqref{eq:sf-smooth-tangential} the following operator
\[
 \mathcal K_{G,d,c}w
 :=\mathcal Q\bigl[\operatorname{div}(G\nabla w)
             +d\cdot\nabla w+c\cdot\nabla\mathcal Sw\bigr]\,.
\]
Lemma~\ref{lem:sf-tame-inverse} below proves that the boundary value problem
\begin{equation}\label{eq:sf-projected-elliptic}
 \mathcal K_{G,d,c}w
 =g,
 \qquad w|_{\partial D}=0,\quad \mathcal Pw=0
\end{equation}
 has a unique smooth solution for any $g\in\mathcal G_\infty$, and a smooth tame inverse whenever $G-I,d,c$ are sufficiently small in a fixed low Sobolev norm. In particular, it does not require smallness in every high norm.

The formulas~\eqref{eq:sf-smooth-coefficients}--\eqref{eq:sf-smooth-E} and the standard Moser-type tame estimates for products and smooth
nonlinear compositions (see~\cite[Chapter~13, Section~3]{TaylorPDEIII})
allow us to obtain, on a sufficiently small fixed Sobolev neighborhood, the following estimate
\begin{equation}\label{eq:sf-geometric-coefficient-estimate}
 \|G-I\|_{H^{s+2}}+\|d\|_{H^{s+2}}+\|c\|_{H^{s+2}}
 \le C_s\bigl(|\varepsilon|+
                  \|\gamma-\mathrm{id}\|_{H^{s+6}}\bigr),\qquad s\ge0.
\end{equation}
In view of Lemma~\ref{lem:sf-tame-inverse} we need to take $s\geq 8$ and the neighborhood can be chosen by imposing
$$\|\gamma-\mathrm{id}\|_{H^{16}}+|\varepsilon|<\delta$$ with a single sufficiently small $\delta>0$. All norms here are on the disk $D$; fixed smooth functions and matrices are included in the constants.

Given $g\in\mathcal G_\infty$, solve~\eqref{eq:sf-projected-elliptic} and define
\begin{equation}\label{eq:sf-smooth-right-inverse}
 \mathcal V(\gamma,\varepsilon)g
 :=D\gamma\,J\nabla\mathcal S\mathcal K_{G,d,c}^{-1}g.
\end{equation}
Notice that in terms of the function $\beta$, this is precisely $D\gamma\, J\nabla \beta$, and hence it is tangent to $\mathcal D_{\mathrm{vol}}$ at $\gamma$. Moreover, we obtain that
\[
 D_\gamma\Phi(\gamma,\varepsilon)\big[
          \mathcal V(\gamma,\varepsilon)g\big]=g\,.
\]
Lemma~\ref{lem:sf-tame-inverse} and the estimate~\eqref{eq:sf-geometric-coefficient-estimate} imply, for $s\ge8$,
\begin{equation}\label{eq:sf-geometric-right-inverse-estimate}
 \|\mathcal V(\gamma,\varepsilon)g\|_{H^{s+1}}
 \le C_s\left[\|g\|_{H^s}
    +\bigl(1+\|\gamma-\mathrm{id}\|_{H^{s+6}}\bigr)\|g\|_{H^8}\right].
\end{equation}
The full map $(\gamma,\varepsilon,g)\mapsto\mathcal V(\gamma,\varepsilon)g$ is smooth tame. Likewise, it is easy to check that the map $\Phi(\gamma,\varepsilon)$ is smooth tame as well. Indeed, its expression is the fixed tame projection $\mathcal Q$ applied to $E_{\gamma,\varepsilon}$, which is given by Equation~\eqref{eq:sf-smooth-E}; this latter expression involves only finitely many derivatives, smooth pointwise matrix inversion, multiplication and exponentials. Tame product estimates put at most one factor in a high norm; all remaining factors are controlled in fixed low norms.

We now apply the Nash--Moser theorem for right inverses~\cite[Part~III, Theorem~1.1.3]{Hamilton1982}. To include the parameter~$\varepsilon$, consider
\[
 \mathcal H(\gamma,\varepsilon)
 :=(\Phi(\gamma,\varepsilon),\varepsilon).
\]
Using the operator $\mathcal V(\gamma,\varepsilon)$ and the fact that $\partial_\varepsilon \Phi(\gamma,\varepsilon)\in \mathcal G_\infty$, we easily infer that the differential of $\mathcal H$ has the smooth tame right inverse
\[
 (g,a)\in \mathcal G_\infty\times \mathbb R\longmapsto
 \left(\mathcal V(\gamma,\varepsilon)
       [g-a\partial_\varepsilon\Phi(\gamma,\varepsilon)],\,a\right)\in T_\gamma \mathcal D_{\mathrm{vol}}\times \mathbb R\,,
\]
which is enough to apply the Nash-Moser theorem. Since $\mathcal H(\mathrm{id},0)=(0,0)$, the theorem provides a solution to the equation 
\[
\mathcal H(\gamma,\varepsilon)=(0,\varepsilon)
\]
for $\varepsilon\neq 0$ small enough.
This gives a smooth family $\widetilde\gamma_\varepsilon\in\mathcal D_{\mathrm{vol}}$, with $\widetilde\gamma_0=\mathrm{id}$ and
\[
 \mathcal Q E_\varepsilon(\widetilde\gamma_\varepsilon-\mathrm{id})=0
 \qquad (|\varepsilon|<\varepsilon_0),
\]
for a single $\varepsilon_0>0$ independent of the regularity index. Here we have used the definition of $\Phi(\gamma,\varepsilon)$. It then follows that the function $E_\varepsilon(\widetilde\gamma_\varepsilon-\mathrm{id})$ is radial.

Setting $\widetilde v_\varepsilon:=\widetilde\gamma_\varepsilon-\mathrm{id}$, we already have a smooth area-preserving diffeomorphism of the disk onto itself and a radial $E_\varepsilon(\widetilde v_\varepsilon)$. However, the normalization $\mathcal P\operatorname{curl}_2\widetilde v_\varepsilon=0$ need not hold. We now impose this condition, which will also determine the first variation uniquely in Section~\ref{S:proof}, without changing the function $\Psi_{\widetilde\gamma_\varepsilon}$. To this end, define the equivariant average of $\widetilde\gamma_\varepsilon$,
\[
 \overline\gamma_\varepsilon(q)
 :=\frac1{2\pi}\int_0^{2\pi}R_{-t}\widetilde\gamma_\varepsilon(R_tq)\,dt\,.
\]
By construction, the map $\overline \gamma_\varepsilon$ commutes with the rotations. Accordingly, using polar coordinates we obtain
\[
\overline \gamma_\varepsilon(q)=R_\theta \overline \gamma_\varepsilon (r,0)\,,
\]
where $\overline \gamma_\varepsilon (r,0)$ denotes the value of $\overline \gamma_\varepsilon$ on the positive $x$-axis. It is then easy to check that this average can be written as
\[
 \overline\gamma_\varepsilon(q) =a_\varepsilon(|q|^2)q+b_\varepsilon(|q|^2)Jq
\]
for smooth functions $a_\varepsilon,b_\varepsilon$ on $[0,1]$, smooth in the small parameter $\varepsilon$, with $a$ close to $1$ and $b$ close to $0$. Smoothness at the origin follows by restricting the average to the $x$-axis: both components are odd smooth functions of $x$, so division by $x$ gives even smooth functions, to which Whitney's theorem~\cite{Wh} applies. Next, set
\[
 \tau_\varepsilon(r^2):=-\arctan\frac{b_\varepsilon(r^2)}{a_\varepsilon(r^2)},\qquad
 \rho_\varepsilon(q):=R_{\tau_\varepsilon(|q|^2)}q,\qquad
 \gamma_\varepsilon:=\widetilde\gamma_\varepsilon\circ\rho_\varepsilon.
\]
The map $\rho_\varepsilon$ is a smooth area-preserving diffeomorphism of the disk, it is equivariant and preserves the function $\Psi_0$. Moreover, the equivariant average of $\gamma_\varepsilon$ is
$$\sqrt{a_\varepsilon(|q|^2)^2+b_\varepsilon(|q|^2)^2}\,q\,,$$
where we have used that the equivariant average of the composition $\widetilde\gamma_\varepsilon\circ\rho_\varepsilon$ is the composition of the average $\overline \gamma_\varepsilon$ with $\rho_\varepsilon$.
In particular, $\operatorname{curl}_2$ applied to this average is zero. Since $\operatorname{curl}_2$ intertwines equivariant averaging with $\mathcal P$, we infer that
\[
 \mathcal P\operatorname{curl}_2(\gamma_\varepsilon-\mathrm{id})=0.
\]
This gives a smooth family $\gamma_\varepsilon$ of area-preserving diffeomorphisms of the disk onto itself, with $\gamma_0=\mathrm{id}$ and
\[
 v_\varepsilon:=\gamma_\varepsilon-\mathrm{id}
 \in C^\infty(\overline D;\mathbb R^2).
\]
The stream function is unchanged, because $\Psi_0\circ\rho_\varepsilon^{-1}=\Psi_0$. Moreover,
\[
 E_\varepsilon(v_\varepsilon)
 =E_\varepsilon(\widetilde v_\varepsilon)\circ\rho_\varepsilon
 =E_\varepsilon(\widetilde v_\varepsilon),
\]
where the first equality follows from $E_\varepsilon(\gamma-\mathrm{id})=(L_\varepsilon\Psi_\gamma)\circ\gamma$, and the second uses the radiality of $E_\varepsilon(\widetilde v_\varepsilon)$ and the fact that $\rho_\varepsilon$ preserves $|q|$. Consequently, the normalized family satisfies
\begin{equation}\label{fZERO}
 \begin{aligned}
  \det(I+Dv_\varepsilon)&=1 &&\text{in }D,\\
  \mathcal Q E_\varepsilon(v_\varepsilon)&=0 &&\text{in }D,\\
  |q+v_\varepsilon(q)|^2&=1 &&\text{on }\partial D,\\
  \mathcal P\operatorname{curl}_2v_\varepsilon&=0 &&\text{in }D.
 \end{aligned}
\end{equation}

Next observe that the second equation in  \eqref{fZERO}
$$
 \mathcal Q E_\varepsilon(v_\varepsilon)=0
$$
says that
$a_\varepsilon:=E_\varepsilon(v_\varepsilon)$ is a radial function on the disk $D$. Therefore, it is of the form
\[
 a_\varepsilon(q)=f_\varepsilon(|q|^2/2)\,,
\]
for some real-valued function $f_\varepsilon$. As we observed above, since $\det(D\gamma_\varepsilon)=1$, setting $\Psi_\varepsilon:=\Psi_0\circ \gamma_\varepsilon^{-1}$, we infer that
$$
 L_\varepsilon\Psi_\varepsilon
                    =f_\varepsilon(\Psi_\varepsilon)\,.$$
It remains to check the regularity of the function $f_\varepsilon$. The function $h_\varepsilon(t):=a_\varepsilon(t,0)$ is even and of class $C^\infty([-1,1])$. Whitney's theorem~\cite{Wh} then implies that there is a function $f_\varepsilon\in C^\infty([0,1/2])$ such that $h_\varepsilon(t)=f_\varepsilon(t^2/2)$. This function $f_\varepsilon$ is the same as the one defined above. The construction is smooth in $\varepsilon$, and $f_0=2$; hence $f_\varepsilon$ is positive for all sufficiently small $\varepsilon$.

To complete the proof of the proposition, we notice that being $\gamma_\varepsilon$ an area-preserving diffeomorphism, it is  symplectic for the standard symplectic form $dx\wedge dy$, proving~\eqref{eq:sf-hamiltonian-conjugacy}.
The unique minimum and the period statement follow from those of
$\Psi_0$ and $X_{\Psi_0}=\mathcal R$.
\end{proof}

The following technical lemma supplies the uniform inverse used in the smooth construction of Proposition~\ref{prop:sf-existence}. In its statement and proof, $\|\cdot\|_{H^s}$ denotes the Sobolev norm on $D$. The particular fixed low indices are only chosen for convenience.

\begin{lm}\label{lem:sf-tame-inverse}
There is $\delta_*>0$ with the following property. Let $G$ be a smooth symmetric matrix field and let $d,c$ be smooth vector fields on $\overline D$ such that
\begin{equation}\label{eq:sf-coefficient-neighborhood}
 \|G-I\|_{H^{12}}+\|d\|_{H^{12}}+\|c\|_{H^{12}}<\delta_*.
\end{equation}
Then the operator
\[
 \mathcal K_{G,d,c}w
 =\mathcal Q\bigl[\operatorname{div}(G\nabla w)
               +d\cdot\nabla w+c\cdot\nabla\mathcal Sw\bigr]
\]
is a bijection from
\[
 \{w\in C^\infty(\overline D): w|_{\partial D}=0,\ \mathcal Pw=0\}
 \quad\text{onto}\quad \mathcal G_\infty.
\]
The map $(G,d,c,g)\mapsto\mathcal K_{G,d,c}^{-1}g$ is smooth tame on the single neighborhood~\eqref{eq:sf-coefficient-neighborhood}. For each integer $s\ge8$, the inverse $w=\mathcal K_{G,d,c}^{-1}g$ satisfies
\begin{equation}\label{eq:sf-tame-elliptic-estimate}
 \|w\|_{H^{s+2}}
 \le C_s\left[\|g\|_{H^s}
                 +(1+M_s)\|g\|_{H^8}\right],\qquad
 M_s:=\|G-I\|_{H^{s+2}}+\|d\|_{H^{s+2}}+\|c\|_{H^{s+2}}.
\end{equation}
The constants $C_s$ are uniform on this neighborhood; the radius $\delta_*$ does not depend on $s$.
\end{lm}

\begin{proof}
We first observe that $\mathcal P$ is an orthogonal projection in $L^2(D)$, preserves $H_0^1(D)$, and commutes with the Dirichlet Laplacian. Moreover, the map $\mathcal S$ preserves $H_0^1(D)\cap\ker\mathcal P$ and is bounded on every $H^s(D)$. In particular,
\begin{equation}\label{eq:sf-S-energy}
 \|\mathcal Sw\|_{L^2}\le\|w\|_{L^2},\qquad
 \|\nabla\mathcal Sw\|_{L^2}\le\|\nabla w\|_{L^2}
 \quad (\mathcal Pw=0).
\end{equation}
Indeed, on the angular mode $e^{im\theta}$, $m\ne0$, it multiplies by $1/(im)$, and
\[
 \|\nabla\mathcal Sw\|_{L^2}^2
 =2\pi\sum_{m\ne0}\frac1{m^2}\int_0^1
       \left(|w_m'(r)|^2+\frac{m^2}{r^2}|w_m(r)|^2\right)r\,dr
 \le\|\nabla w\|_{L^2}^2.
\]
The formulas first hold for smooth functions and extend by density.

\medskip\noindent\emph{Step~1: the weak inverse and a fixed regularity estimate.}
Set
\[
 V:=H_0^1(D)\cap\ker\mathcal P,\qquad
 \|w\|_V:=\|\nabla w\|_{L^2}.
\]
The weak form of the equation $\mathcal K_{G,d,c} w=g$ is
\[
a(w,\varphi)=-\int_D g\varphi
\]
for all $\varphi \in V$, with 
\[
 a(w,\varphi):=\int_DG\nabla w\cdot\nabla\varphi
       -\int_D(d\cdot\nabla w)\varphi
       -\int_D(c\cdot\nabla\mathcal Sw)\varphi\,.
\]
There is no extra term from $\mathcal Q$, since $\mathcal Q\varphi=\varphi$ provided that $\mathcal P \varphi=0$. If $\lambda>0$ is a lower ellipticity bound for $G$ and $C_P$ is the Poincar\'e constant of the disk, then~\eqref{eq:sf-S-energy} gives
\[
 a(w,w)\ge
 \bigl[\lambda-C_P(\|d\|_{L^\infty}+\|c\|_{L^\infty})\bigr]
                 \|\nabla w\|_{L^2}^2.
\]
Choose $\delta_*$ small enough that the bracket is at least $1/2$ (in particular, this justifies the assumption that $G$ is uniformly positive definite, with lower bound $\lambda>0$). The weak form is continuous and coercive, so Lax--Milgram gives a unique weak solution $w\in V$ for every $g\in \mathcal G_\infty$, with
\begin{equation}\label{eq:sf-weak-bound}
 \|w\|_{H^1}\le C\|g\|_{L^2}\,,
\end{equation}
the smoothness of $w$ for smooth data $g$ will follow from the regularity argument below.

We will also use one fixed higher regularity norm. The Dirichlet Laplacian is an isomorphism
\[
 \Delta_D:\{w\in H^{10}(D)\cap H_0^1(D):\mathcal Pw=0\}
       \longrightarrow H^8(D)\cap\ker\mathcal P.
\]
Sobolev product estimates and the boundedness of $\mathcal Q,\mathcal S$ imply
\[
 \|(\mathcal K_{G,d,c}-\Delta_D)w\|_{H^8}
 \le C\bigl(\|G-I\|_{H^{12}}+\|d\|_{H^{12}}+\|c\|_{H^{12}}\bigr)
                  \|w\|_{H^{10}}.
\]
Shrinking $\delta_*$ once at this fixed regularity, the operator 
\[
B:=\Delta_D^{-1}(\mathcal K_{G,d,c}-\Delta_D) :\{w\in H^{10}(D)\cap H_0^1(D):\mathcal Pw=0\} \longrightarrow \{w\in H^{10}(D)\cap H_0^1(D):\mathcal Pw=0\}
\]
has norm smaller than $1/2$. Therefore, for all $g\in H^8(D)\cap\ker\mathcal P$, 
the Neumann series
\[
w_N:=\sum_{j=0}^\infty (-B)^j\Delta_D^{-1}g
\]
converges and it solves the equation $\mathcal K_{G,d,c} w_N=g$ with the estimate 
\begin{equation}\label{eq:sf-base-bound}
 \|w_N\|_{H^{10}}\le C\|g\|_{H^8}.
\end{equation}
This solution $w_N$ agrees with the previously obtained weak solution $w$ by uniqueness. No higher regularity is used to choose $\delta_*$. In what follows, we use the notation $w$ for the solution in $H^{10}(D)\cap H_0^1(D)\cap \mathcal \ker\mathcal P$ that we have obtained.

\medskip\noindent\emph{Step~2: interior regularity, including the center of the disk.}
The calculations in Steps~2--4 are first a priori estimates for smooth $w$; their application to the $H^{10}$ solution will be justified by regularization at the end of Step~4.

For a covariant $k$-tensor field $U$, let
\[
 (U_t^{(k)}U)(q)
 :=(R_t^T)^{\otimes k}U(R_tq).
\]
Let $\mathcal P_k,\mathcal Q_k$ be the average and complementary projection for this action, and let $\mathcal S_k$ be its zero-mean angular inverse (we defined these operators before for scalar functions, i.e., $k=0$). These are bounded on tensor-valued $L^2$; on tensors in $\ker \mathcal P_k$, it readily follows that $\|\mathcal S_k\|_{L^2\to L^2}\le1$. Equivariance of differentiation gives
\begin{equation}\label{eq:sf-tensor-intertwining}
 D^k\mathcal P=\mathcal P_kD^k,\qquad
 D^k\mathcal Q=\mathcal Q_kD^k,\qquad
 D^k\mathcal S=\mathcal S_kD^k
 \quad\text{on }\ker\mathcal P.
\end{equation}
Consequently $\mathcal P_kD^kw=0$ (recall that $\mathcal P w=0$). Also observe that multiplication by a radial cutoff commutes with $\mathcal P_k$.

More precisely, differentiation also gives $\nabla\mathcal S_kU=\mathcal S_{k+1}\nabla U$ and $\mathcal P_{k+1}\nabla U=0$ when $\mathcal P_kU=0$. Thus, for a radial cutoff $\chi$,
\[
 \|\chi\nabla\mathcal S_kU\|_{L^2}
 =\|\mathcal S_{k+1}(\chi\nabla U)\|_{L^2}
 \leq\|\chi\nabla U\|_{L^2}.
\]
Here the tensor action is unitary, so $\mathcal Q_k$ is self-adjoint; in particular, $\langle\mathcal Q_kF,\chi^2D^kw\rangle=\langle F,\chi^2D^kw\rangle$.

Take a radial cutoff $\chi$ compactly supported in $D$ and equal to one on the disk $D_{3/4}:=\{q\in D: |q|<3/4\}$. In order to prove an interior a priori bound for the solutions, let us assume that $w$ solves $\mathcal K_{G,d,c} w=g$.  Differentiate $k$ times the equation $\mathcal K_{G,d,c} w=g$ and test it against the square of this cutoff function times $D^kw$. By~\eqref{eq:sf-tensor-intertwining}, the projection $\mathcal Q_k$ can be omitted from the pairing (because $\mathcal P_k (\chi^2 D^k w)=0$). After integration by parts, the highest-order term is the positive form
\[
 \int_D\chi^2\sum_{|I|=k}G\nabla D_Iw\cdot\nabla D_Iw.
\]
The coefficients in this term are not differentiated. The remaining principal commutators are products containing at least one derivative of $G$, and obey the tame estimate
\begin{equation}\label{eq:sf-commutator}
 \|[D^k,G]\nabla w\|_{L^2}
 \le C_k\bigl(\|G-I\|_{H^k}\|w\|_{H^5}
               +\|G-I\|_{H^5}\|w\|_{H^k}\bigr),\qquad k\ge5.
\end{equation}
This easily follows by the Leibniz rule and Sobolev interpolation. The analogous estimates apply to the vector fields $d,c$. For the nonlocal term $c\cdot \nabla \mathcal Sw$, use $D^k\mathcal Sw=\mathcal S_kD^kw$ and~\eqref{eq:sf-S-energy} in its tensor form (i.e., the same estimates hold using $\mathcal S_k$ on tensors instead of $\mathcal S$); radial cutoffs also commute with these angular operators. Terms involving derivatives of the cutoff are of lower differential order.

To spell out the energy estimate, set $U=D^kw$, $A_k=\|\chi\nabla U\|_{L^2}$ and $C_k=[D^k,G]\nabla w$. The principal part is $\operatorname{div}(G\nabla U)+\operatorname{div}C_k$.

Its commutator contributes at most
\[
 |\langle\operatorname{div}C_k,\chi^2U\rangle|
 \leq C\|C_k\|_{L^2}(A_k+\|U\|_{L^2}).
\]
Leibniz' rule similarly writes the two first-order terms as
$d\cdot\nabla U+[D^k,d]\nabla w$ and
$c\cdot\nabla\mathcal S_kU+[D^k,c]\nabla\mathcal Sw$.
The commutators satisfy~\eqref{eq:sf-commutator}, with the corresponding coefficients, while the undifferentiated nonlocal term obeys
\[
 \left|\int_D\chi^2(c\cdot\nabla\mathcal S_kU):U\right|
 \leq\|c\|_{L^\infty}A_k\|\chi U\|_{L^2}.
\]
The local term with $d$ has the same bound. The cutoff term from the principal part is bounded by $C\|G\|_{L^\infty}A_k\|U\|_{L^2}$. Hence none of these terms requires smallness of a high coefficient norm.

With $k=s+1$, it is clear that the differentiated datum $D^k g$ has $H^{-1}$ norm bounded by $C_s\|g\|_{H^s}$. Indeed, writing a multiindex of length $s+1$ as $\alpha=\beta+e_i$, one has $\langle D^\alpha g,Z\rangle=-\int_DD^\beta g\,\partial_iZ$, so only $s$ derivatives of $g$ are needed. In particular,
\[
 |\langle D^{s+1}g,\chi^2U\rangle|
 \leq C_s\|g\|_{H^s}(A_{s+1}+\|w\|_{H^{s+1}}).
\]
Denote the right-hand quantity to be estimated by
$\mathfrak B_s:=\|g\|_{H^s}+\|w\|_{H^{s+1}}+(1+M_s)\|w\|_{H^5}$.
The preceding bounds and ellipticity give
$\lambda A_{s+1}^2\leq C_s\mathfrak B_s(A_{s+1}+\mathfrak B_s)$.
Young's inequality absorbs half the left-hand side and yields $A_{s+1}\leq C_s\mathfrak B_s$.

Cauchy--Schwarz and the preceding estimates therefore yield the interior a priori bound
\begin{equation}\label{eq:sf-interior-tame}
 \|w\|_{H^{s+2}(D_{3/4})}
 \le C_s\left[\|g\|_{H^s}+\|w\|_{H^{s+1}}
                     +(1+M_s)\|w\|_{H^5}\right],\qquad s\ge8.
\end{equation}
This argument uses smooth Cartesian tensors at the origin, not a singular polar frame. This a priori bound, together with a tangential one in the next step, will allow us to show that the solutions $w$ we found before is, in fact, of class $C^\infty$.

\medskip\noindent\emph{Step~3: tangential estimates at the boundary.}
The vector field $\mathcal R$ is tangent to $\partial D$ and commutes with $\mathcal P,\mathcal Q,\mathcal S$. In particular,
\[
 \mathcal R^kw|_{\partial D}=0,\qquad \mathcal P\mathcal R^kw=0.
\]
Apply $\mathcal R^k$ to the equation $\mathcal K_{G,d,c} w=g$, which yields
\[
\mathcal K_{G,d,c} (\mathcal R^k w) = \mathcal R^k g - [\mathcal R^k,\mathcal K_{G,d,c}]w\,.
\]
Set $u=\mathcal R^kw$ and $F_k=\mathcal R^kg-[\mathcal R^k,\mathcal K_{G,d,c}]w$. Since $u\in V$ in this a priori calculation, the weak equation and coercivity give
\[
 \tfrac12\|u\|_V^2\leq a(u,u)=-\langle F_k,u\rangle
 \leq\|F_k\|_{V^*}\|u\|_V,
 \qquad \|u\|_{H^1}\leq C\|F_k\|_{H^{-1}}.
\]
For clarity, the first commutator is
\[
 [\mathcal R,\mathcal K_{G,d,c}]w
 =\mathcal Q\bigl[\operatorname{div}(\widehat G\nabla w)
       +\widehat d\cdot\nabla w+\widehat c\cdot\nabla\mathcal Sw\bigr],
 \quad
 \begin{cases}
 \widehat G=\mathcal RG+GJ-JG,\\
 \widehat d=\mathcal Rd-Jd,\quad\widehat c=\mathcal Rc-Jc.
 \end{cases}
\]
Iterating this identity, each principal commutator is a divergence of a product involving at most $k$ derivatives of $w$. For $k=s+1$, its $H^{-1}$ norm is bounded by $C_s(\|w\|_{H^{s+1}}+M_s\|w\|_{H^5})$. Also $\|\mathcal R^{s+1}g\|_{H^{-1}}\leq C_s\|g\|_{H^s}$, by transferring one angular derivative to the test function.

Now, we can use the bilinear form $a$ (see Step~1) for $\mathcal R^kw$. Commuting $\mathcal R$ with gradient and divergence only introduces fixed zeroth-order matrices, and the commutators with $G,d,c$ satisfy the same tame product estimates as~\eqref{eq:sf-commutator}. In divergence form the commutator with the second-order part contains at most $k$ derivatives of $w$ before the final divergence. Thus, with $k=s+1$, as before we obtain the estimate
\begin{equation}\label{eq:sf-tangential-tame}
 \|\mathcal R^{s+1}w\|_{H^1}
 \le C_s\left[\|g\|_{H^s}+\|w\|_{H^{s+1}}
                     +(1+M_s)\|w\|_{H^5}\right].
\end{equation}
On $1/2<r<1$, this controls all derivatives of order $s+2$ containing at most one radial derivative. Lower tangential orders are controlled by the lower norm on the right.

\medskip\noindent\emph{Step~4: recovery of normal derivatives.}
Since the interior a priori estimates were obtained in Step~2, now we can focus on the collar $1/2<r<1$, which we parametrize by polar coordinates. There the equation $\mathcal K_{G,d,c} w=g$ has the form
\begin{equation}\label{eq:sf-collar-equation}
 \mathcal Q\bigl[a\,\partial_{rr}w+b\,\partial_{r\theta}w+e\,\partial_{\theta\theta}w
            +b_1\partial_rw+b_2\partial_\theta w
            +c_1\partial_r(\mathcal Sw)+c_2\partial_\theta(\mathcal Sw)\bigr]=g,
\end{equation}
for some factors $a,b,e,b_1,b_2,c_1,c_2$, where $a=e_r^TGe_r\ge\lambda>0$. The other coefficients and their Sobolev norms are controlled by $G,d,c$ and one derivative of $G$, with the harmless factors $r^{-1}$ bounded on the collar.

For each fixed $r$, the operator $u\mapsto\mathcal Q(au)$ on zero-mean functions of $\theta$ has the explicit inverse
\begin{equation}\label{eq:sf-normal-inverse}
 \mathcal I_a f
 =\frac1a\left(f-\frac{\mathcal P(f/a)}{\mathcal P(1/a)}\right),
 \qquad \mathcal Pf=0\,,
\end{equation}
which is obvious because $\mathcal P\mathcal I_af=0$ and $\mathcal Q(a\mathcal I_af)=f$. The estimate $a>\lambda>0$ on the function $a$ gives the uniform estimate
\[
 \|\mathcal I_af\|_{L^2(\mathbb S^1)}
 \le\lambda^{-1}\|f\|_{L^2(\mathbb S^1)}.
\]
For this last bound, if $u=\mathcal I_af$, then $\mathcal Pu=0$ and
$\lambda\|u\|_{L^2}^2\leq\int au^2=\langle\mathcal Q(au),u\rangle=\langle f,u\rangle$.
The formula itself follows by writing $au=f+\mu(r)$ and imposing $\mathcal Pu=0$, which gives $\mu=-\mathcal P(f/a)/\mathcal P(1/a)$.

All derivatives of $a^{-1}$ and of the denominator in~\eqref{eq:sf-normal-inverse} satisfy the usual tame composition estimates; the denominator is bounded away from zero when $a$ is bounded above and below.

Let $m=s+2$. Differentiate~\eqref{eq:sf-collar-equation} by
$\partial_r^{j-2}\partial_\theta^{m-j}$, for $j=2,\ldots,m$. Since $\mathcal P$ commutes with these derivatives on the collar, the unknown
$u=\partial_r^j\partial_\theta^{m-j}w$ has zero angular mean and appears as 
$$
\mathcal Q(au)=f_j\,,
$$
where the function $f_j$ contains terms of total order $m$ with derivatives $\partial_r$ up to order $j-1$, and terms of total order $\leq m-1$. The inverse defined by Equation~\eqref{eq:sf-normal-inverse} allows us to obtain $u$ from this equation. Since the other derivatives of total order $m$ have at most $j-1$ radial derivatives, they have already been estimated, beginning with~\eqref{eq:sf-tangential-tame}. If a derivative falls on a coefficient, the corresponding derivative of $w$ has total order at most $m-1$. The terms containing $\mathcal Sw$ also have this lower order, since $\partial_r,\partial_\theta$ commute with $\mathcal S$ on the collar. More explicitly, on the collar $A=\{1/2<r<1\}$, set $u_\ell=\partial_r^\ell\partial_\theta^{m-\ell}w$ and $D_j=\partial_r^{j-2}\partial_\theta^{m-j}$. Leibniz' rule gives
\[
 f_j=D_jg-\mathcal Q(bu_{j-1}+eu_{j-2}+C_j),\qquad
 \|C_j\|_{L^2(A)}\leq C_s\bigl[\|w\|_{H^{m-1}}+(1+M_s)\|w\|_{H^5}\bigr].
\]
Here $C_j$ includes all terms in which a derivative hits a leading coefficient and all differentiated first-order terms. For instance, when $m>j$ it includes $(m-j)(\partial_\theta a)\partial_r^j\partial_\theta^{m-j-1}w$: this has $j$ radial derivatives but total order $m-1$, so is already controlled. Since $\mathcal Q$ is an $L^2(A)$ contraction, the inverse estimate gives
\[
 \|u_j\|_{L^2(A)}\leq\lambda^{-1}\bigl(\|D_jg\|_{L^2(A)}
       +\|b\|_\infty\|u_{j-1}\|_{L^2(A)}
       +\|e\|_\infty\|u_{j-2}\|_{L^2(A)}+\|C_j\|_{L^2(A)}\bigr).
\]
Thus $u_0,u_1$ start the induction by Step~3, and each subsequent $u_j$ is bounded by $C_s\mathfrak B_s$. Only $m-2=s$ derivatives of $g$ occur.

Product interpolation, as in~\eqref{eq:sf-commutator}, then gives the a priori bound
\begin{equation}\label{eq:sf-collar-tame}
 \|w\|_{H^{s+2}(\{1/2<r<1\})}
 \le C_s\left[\|g\|_{H^s}+\|w\|_{H^{s+1}}
                     +(1+M_s)\|w\|_{H^5}\right].
\end{equation}
Only fixed low coefficient bounds multiply the highest derivatives recovered in this triangular induction; high coefficient norms occur linearly multiplying a low norm of $w$.

Let us now use all the a priori regularity bounds that we have established to show that the solution $w\in H^{10}$ we obtained in Step~1 is $C^\infty$. Indeed, inductively assume $w\in H^{s+1}$ and $g\in H^s$ (in fact, $g$ is smooth). In the interior, use radial Friedrichs mollifiers $J_\eta$ after the radial cutoff; they preserve the tensor averaging relations~\eqref{eq:sf-tensor-intertwining}. Here $J_\eta U=\rho_\eta*U$, where $\rho_\eta(q)=\eta^{-2}\rho(q/\eta)$ and $\rho$ is smooth, radial, supported in the unit disk and has integral one. The localized tensor $U=\chi D^{s+1}w$ is extended by zero; for sufficiently small $\eta$, $J_\eta U$ is smooth and compactly supported in $D$.

The commutator estimate
\[
 \|[J_\eta,a]\partial_i u\|_{L^2}
 \le C\|a\|_{W^{1,\infty}}\|u\|_{L^2}
\]
is uniform in the mollification scale $\eta>0$ and justifies the integrations by parts.

In fact, after localization the differentiated equation is
$\mathcal Q_k\operatorname{div}(G\nabla U)=F_\chi$ with
$\|F_\chi\|_{H^{-1}}\leq C_s\mathfrak B_s$; all cutoff errors and first-order terms belong to $H^{-1}$ using only $U\in L^2$. Extending the coefficients smoothly near its support, regularization gives
\[
 \mathcal Q_k\operatorname{div}(G\nabla J_\eta U)
 =J_\eta F_\chi+\mathcal Q_k\operatorname{div}([G,J_\eta]\nabla U).
\]
The displayed commutator bound controls the last flux in $L^2$ independently of $\eta$. Pairing with $J_\eta U$, using $\mathcal Q_kJ_\eta U=J_\eta U$ and ellipticity, yields $\|J_\eta U\|_{H^1}\leq C_s\mathfrak B_s$. No unknown derivative of $U$ is used on the right.

 At the boundary, angular difference quotients preserve the zero trace and the zero angular average. They give~\eqref{eq:sf-tangential-tame} by passage to the limit. Formula~\eqref{eq:sf-normal-inverse}, applied to the distributional equation, then recovers the missing normal derivatives.

The energy estimates are uniform in the regularization parameters. Weak compactness in $H^1$, together with strong convergence in $L^2$, therefore proves the extra derivative rather than assuming it. For angular differences one uses $(u\circ R_h-u)/h$, which is admissible for each $h\ne0$. Taking first $s=9$ yields $H^{10}\to H^{11}$ globally; repeating with $s=10,11,\ldots$ gives the bootstrap. Starting from~\eqref{eq:sf-base-bound}, these steps show that smooth data and coefficients give $w\in H^s$ for every $s$, hence $w\in C^\infty(\overline D)$.

\medskip\noindent\emph{Step~5: tame estimates and dependence on the coefficients.}
Combining~\eqref{eq:sf-interior-tame} and~\eqref{eq:sf-collar-tame} gives
\[
 \|w\|_{H^{s+2}}
 \le C_s\left[\|g\|_{H^s}+\|w\|_{H^{s+1}}
                     +(1+M_s)\|w\|_{H^5}\right].
\]
For each $s$, interpolate the lower derivative term as
\[
 \|w\|_{H^{s+1}}
 \le\eta\|w\|_{H^{s+2}}+C_{s,\eta}\|w\|_{H^5}.
\]
Choose $\eta$ small enough to absorb this term and use~\eqref{eq:sf-base-bound}. This proves~\eqref{eq:sf-tame-elliptic-estimate} in the statement of the lemma. The interpolation parameter may depend on $s$; the coefficient neighborhood~\eqref{eq:sf-coefficient-neighborhood} does not. In particular, this argument never requires $C_s\delta_*<1$ at the higher levels.

Finally write $p=(G,d,c)$. Write $R_p=\mathcal K_p^{-1}$ and $|h|_r=\|h_G\|_{H^r}+\|h_d\|_{H^r}+\|h_c\|_{H^r}$ for a coefficient variation $h=(h_G,h_d,h_c)$. The derivative of the operator is
\[
 T_hu:=D_p\mathcal K_p[h]u
 =\mathcal Q[\operatorname{div}(h_G\nabla u)+h_d\cdot\nabla u
                       +h_c\cdot\nabla\mathcal Su].
\]
The product estimates give $\|T_hu\|_{H^s}\leq C_s(|h|_{s+2}\|u\|_{H^{10}}+|h|_{12}\|u\|_{H^{s+2}})$ and $\|T_hu\|_{H^8}\leq C|h|_{12}\|u\|_{H^{10}}$. Moreover,
\[
 R_{p'}-R_p=-R_{p'}T_{p'-p}R_p.
\]
Applied with~\eqref{eq:sf-tame-elliptic-estimate}, this identity proves continuity of $(p,g)\mapsto R_pg$ in every smooth seminorm. Difference quotients in the same identity then give the derivative below.

The fixed-level inverse constructed above depends smoothly on $p$ by the Banach inverse theorem, and the resolvent identity and the regularity estimates give
\[
 D_p(\mathcal K_p^{-1})[\dot p]g
 =-\mathcal K_p^{-1}(D_p\mathcal K_p[\dot p])\mathcal K_p^{-1}g.
\]
Combining the two estimates for $T_h$ with the inverse bounds gives explicitly
\[
 \|D_pR_p[h]g\|_{H^{s+2}}
 \leq C_s\Bigl[|h|_{12}\|g\|_{H^s}
       +\bigl(|h|_{s+2}+(1+M_s)|h|_{12}\bigr)\|g\|_{H^8}\Bigr].
\]
For order $\ell$, the exact formula is
\[
 D_p^\ell R_p[h_1,\ldots,h_\ell]g
 =(-1)^\ell\sum_{\sigma\in\mathfrak S_\ell}
 R_pT_{h_{\sigma(1)}}R_p\cdots T_{h_{\sigma(\ell)}}R_pg.
\]
At each application, the high norm falls on exactly one factor; all the others use the uniform $H^8\to H^{10}$ inverse bound and low coefficient norms. Thus the resulting bounds are sums of a high norm of $g$, of one $h_i$, or of $p$, multiplied by low norms of the other factors. Dependence on $g$ is linear, so this also treats all mixed derivatives.

More precisely, for every $\ell\geq1$, the preceding iteration yields
\[
 \begin{split}
 \|D_p^\ell R_p[h_1,\ldots,h_\ell]g\|_{H^{s+2}}
 \leq C_{s,\ell}\Bigg[&\bigl(\|g\|_{H^s}+(1+M_s)\|g\|_{H^8}\bigr)
                           \prod_{i=1}^{\ell}|h_i|_{12}\\
 &+\|g\|_{H^8}\sum_{i=1}^{\ell}|h_i|_{s+2}
                           \prod_{j\ne i}|h_j|_{12}\Bigg].
 \end{split}
\]
To see why no product of high norms occurs, the low estimate for $R_pT_hu$ uses only $|h|_{12}\|u\|_{H^{10}}$. In the high estimate one therefore either propagates the existing high norm of $u$, or places it on $h$ or $p$ while using this low estimate for everything to its right. This proves the displayed inequality inductively in $\ell$, on the same fixed neighborhood.

For completeness, continuity at each order can be checked without an additional inverse theorem: subtract the two permutation formulas at $p'$ and $p$ and telescope the products. Each summand contains one difference $R_{p'}-R_p=-R_{p'}T_{p'-p}R_p$; differences of the directions are treated by the linearity of $T_h$. The preceding high and low estimates show that every such term tends to zero in $H^{s+2}$ whenever the coefficients, directions and datum converge in the smooth topology. Thus the directional differentials above are jointly continuous at every order, which supplies smoothness in the Fr\'echet topology as well as the tame bounds.

Higher derivatives are finite sums of such compositions, since $\mathcal K_p$ depends affinely on $p$. Estimate~\eqref{eq:sf-tame-elliptic-estimate}, together with the tame product estimates, controls each derivative with a fixed loss of grades independent of the high index. Applying the same estimates to differences proves continuity of these derivatives in the smooth topology. Thus the inverse family is smooth tame, as claimed. This completes the proof of the lemma.
\end{proof}

\section{Proof of Theorem~\ref{prop:sf-existence}}\label{S:proof}

For any small enough $\varepsilon>0$, define the domain
\begin{equation}\label{eq:sf-domain}
 \Omega_\varepsilon
 :=\{(r\cos\phi,r\sin\phi,z):(\log r)^2+z^2<\varepsilon^2, \ \phi \in \mathbb{R}/(2\pi \mathbb{Z})\}.
\end{equation}
It is clear that this is a rotationally symmetric toroidal domain (disjoint from the $z$-axis), which is a tubular neighborhood of the circle $\{r=1,z=0\}$ of thickness $O(\varepsilon)$.

Let $\Psi_\varepsilon$ be the function on the disk obtained in Proposition~\ref{lem:sf-deformation}, which has regularity $C^\infty(\overline D)$. To pass from the $(x,y)\in D$ coordinates to the cylindrical coordinates $(r,z)$, we
introduce the following change of variables on the meridional section:
\begin{equation}\label{eq:sf-coordinates}
 r=e^{\varepsilon x}\,,\qquad z=\varepsilon y\,,
 \end{equation}
which transforms the domain $\Omega_\varepsilon$ onto the torus $D\times \mathbb S^1$. We also define
$$
 \psi(r,z):=\varepsilon^2\Psi_\varepsilon(x,y).
$$
It is clear that
\begin{equation*}
 \partial_r\psi=\frac{\varepsilon}{r}\partial_x \Psi_\varepsilon\,, \qquad
 \partial_z\psi=\varepsilon\partial_y\Psi_\varepsilon\,.
\end{equation*}

The Grad-Shafranov operator
$\Delta^*=\partial_{rr}-r^{-1}\partial_r+\partial_{zz}$ acting on $\psi$ then satisfies
\begin{align*}
 \frac{\Delta^*\psi}{r^2}
 &=e^{-4\varepsilon x}
       \bigl(\partial_{xx}\Psi_\varepsilon
                  -2\varepsilon\partial_x\Psi_\varepsilon\bigr)
       +e^{-2\varepsilon x}\partial_{yy}\Psi_\varepsilon
 =f_\varepsilon(\Psi_\varepsilon)\,,
 \label{eq:sf-gs-transform}
\end{align*}
where we have used the equation that the function $\Psi_\varepsilon$ satisfies on the disk. Then, defining
\begin{equation*}\label{eq:sf-bernoulli}
 H(s):=\varepsilon^2
          \int_0^{s/\varepsilon^2}f_\varepsilon(c)\,dc,
 \qquad 0\leq s\leq\varepsilon^2/2\,,
\end{equation*}
we obtain that $H'(\psi)=f_\varepsilon(\Psi_\varepsilon)$, so
\begin{equation}\label{eq:sf-gs}
 \Delta^*\psi=r^2H'(\psi)\,,
\end{equation}
which is precisely the Grad-Shafranov equation stated in Section~\ref{S:intro} with constant $F$.

Finally, to define the MHS equilibrium, set $\kappa:=n/m$ and write
\begin{equation}\label{eq:sf-velocity-pressure}
 B=-\frac{\partial_z\psi}{r}\mathbf e_r
       +\frac{\kappa}{r}\mathbf e_\phi
       +\frac{\partial_r\psi}{r}\mathbf e_z\,.
\end{equation}
By construction, $B$ is of class $C^\infty$, $\operatorname {div } B=0$ and it is an MHS equilibrium because $\psi$ satisfies the Grad-Shafranov equation. In particular, $B$ is tangent to $\partial\Omega_\varepsilon$ because $\psi|_{\partial\Omega_\varepsilon}=\varepsilon^2/2$ (recall that $\Psi_\varepsilon|_{\partial D}=1/2$). We also notice that
\[
B\times \cu  B
       -\nabla H(\psi)=0\,,
\]
and hence $H(\psi)$ coincides with minus the plasma pressure.

Let us now analyze the orbits of $B$. The first obvious observation is that $B_\phi=\kappa/r\neq0$, so
the field is nowhere zero. In $(x,y,\phi)$ coordinates, the ODEs for the orbits of $B$ are
\begin{equation}\label{eq:sf-trajectory}
 \dot x=-\frac{\partial_y\Psi_\varepsilon}{r^2}\,,\qquad
 \dot y=\frac{\partial_x\Psi_\varepsilon}{r^2}\,,\qquad
 \dot\phi=\frac{\kappa}{r^2}\,, \qquad
 r^2= e^{2\varepsilon x}.
\end{equation}
Thus, as a vector field in these coordinates we get
\begin{equation}\label{eq:sf-rescaling}
 r^2B=X_{\Psi_\varepsilon}+\kappa\partial_\phi\,,
\end{equation}
where $X_{\Psi_\varepsilon}=J\nabla\Psi_\varepsilon$ is the 2D Hamiltonian field associated to the function $\Psi_\varepsilon$ (with respect to the standard symplectic form). Applying the diffeomorphism $\gamma_\varepsilon^{-1}\times\mathrm{id}_{S^1}:D\times \mathbb S^1\to D\times \mathbb S^1$ (recall again Proposition~\ref{lem:sf-deformation}) to the
rescaled field $m r^2B$ we see that
\begin{equation}\label{eq:sf-model}
\big(\gamma_\varepsilon^{-1}\times\mathrm{id}_{S^1}\big)_{*}(mr^2B) = m(-Y\partial_X+X\partial_Y)+n\partial_\phi
 \quad\text{on }D\times S^1,
\end{equation}
where $(X,Y)$ are Cartesian coordinates on the unit disk.
Its flow is the globally defined circle action
\begin{equation}\label{eq:sf-action}
 (X+iY,e^{i\phi})\longmapsto
       \bigl(e^{imt}(X+iY),e^{int}e^{i\phi}\bigr),
 \qquad t\in\mathbb R/(2\pi\mathbb Z).
\end{equation}
This formula remains valid at the disk centre and avoids the polar
coordinate singularity there.

For $X+iY\neq0$, coprimality implies that the minimal period of
\eqref{eq:sf-action} is $2\pi$. During this period the orbit makes
$m$ meridional and $n$ longitudinal circuits. On $X=Y=0$, the minimal
period is $2\pi/n$. The image of
this central fibre is the Euclidean circle
\begin{equation}\label{eq:expanded-core}
 C =\{(r_\varepsilon\cos\phi,r_\varepsilon\sin\phi,z_\varepsilon):
              \phi\in\mathbb R/(2\pi\mathbb Z)\},
 \qquad
 r_\varepsilon=e^{\varepsilon x_\varepsilon},\quad
 z_\varepsilon=\varepsilon y_\varepsilon.
\end{equation}
where we recall that $q_\varepsilon=(x_\varepsilon, y_\varepsilon)=\gamma_\varepsilon(0)$. Since the rescaling factor $m r^2$ is
nonvanishing, $B$ has the same trajectories (up to a change of orientation in the case that $m<0$). Hence all orbits are closed and the orbit foliation has precisely
the asserted Seifert multiplicities.

The level circles of $\Psi_\varepsilon$ revolve into nested invariant
tori surrounding $C$. Orient their meridians and longitudes as stated in Section~\ref{S:intro} (this is independent of $n$ and $m$).
The latter circles have linking number zero with $C$ and hence give
the preferred longitude. A meridian has linking number $-1$
with $C$, so the winding class $m\mu+n\lambda$ proves
\eqref{eq:sf-linking}.

The magnetic axis is the set $\{\psi=0\}$, and it is the common core of the nested toroidal surfaces $\{\psi=c\}$, $0<c\leq\varepsilon^2/2$. Moreover, it is a round (unknotted) circle about the symmetry axis (i.e., the $z$-axis). Let us estimate the departure of the magnetic axis $C$ in \eqref{eq:expanded-core} from the reference circle $$\{r=1,z=0\}\,.$$
To this end, we follow the notation and computations presented in the proof of Proposition~\ref{lem:sf-deformation}. Indeed,
\[
(x_\varepsilon,y_\varepsilon)=q_\varepsilon=\gamma_\varepsilon(0)=v_\varepsilon(0)\,.
\]
The smooth family $v_\varepsilon$ satisfies the geometric equations and normalization~\eqref{fZERO}. Since $v_0=0$, it admits the expansion
\[
 v_\varepsilon=\varepsilon v^{(1)}+O(\varepsilon^2),
 \qquad
 v^{(1)}:=\left.\partial_\varepsilon v_\varepsilon\right|_{\varepsilon=0},
\]
in every fixed $C^k$ norm. We determine $v^{(1)}$ directly by differentiating~\eqref{fZERO}.

First, let us compute the linearization of the function $E_\varepsilon(v)$ at $(v,\varepsilon)=(0,0)$. At this point $W=q$ and $M=I$, so for any smooth vector field $\hat v$ the definition~\eqref{eq:sf-residual} gives
\begin{align*}
 D_vE_0(0)[\hat v]
 &=\operatorname{tr}\bigl(D(-J\mathcal R\hat v)-D\hat v\bigr)\\
 &=\operatorname{curl}_2(\mathcal R\hat v)-\operatorname{div}\hat v\\
 &=\mathcal R\operatorname{curl}_2\hat v.
\end{align*}
In the last equality we used the identity
$\operatorname{curl}_2(\mathcal R\hat v)=\mathcal R\operatorname{curl}_2\hat v+\operatorname{div}\hat v$.
Since an angular derivative has zero rotational average, $\mathcal Q$ acts as the identity on this expression. Also,
\[
 E_\varepsilon(0)
 =e^{-4\varepsilon x}(1-2\varepsilon x)+e^{-2\varepsilon x},
 \qquad
 \left.\partial_\varepsilon E_\varepsilon(0)\right|_{\varepsilon=0}=-8x,
\]
and $\mathcal Qx=x$. Differentiating the four equations in~\eqref{fZERO} therefore yields
\begin{equation}\label{eq:sf-first-variation-system}
 \begin{aligned}
  \operatorname{div}v^{(1)}&=0 &&\text{in }D,\\
  \mathcal R\operatorname{curl}_2v^{(1)}&=8x &&\text{in }D,\\
  v^{(1)}\cdot\nu&=0 &&\text{on }\partial D,\\
  \mathcal P\operatorname{curl}_2v^{(1)}&=0 &&\text{in }D.
 \end{aligned}
\end{equation}

The vector field
\[
 v_*(x,y):=(1-x^2-3y^2,\,2xy)
\]
satisfies this system. Indeed,
\[
 \operatorname{div}v_*=0,\qquad
 \operatorname{curl}_2v_*=8y,\qquad
 \mathcal R(8y)=8x,\qquad \mathcal P(8y)=0,
\]
and on $\partial D$,
\[
 v_*\cdot\nu=x(1-x^2-y^2)=0.
\]
To verify uniqueness in the normalized class, let $u=v^{(1)}-v_*$. Then $\operatorname{div}u=0$, $u\cdot\nu=0$ on $\partial D$, and
\[
 \mathcal R\operatorname{curl}_2u=0,
 \qquad \mathcal P\operatorname{curl}_2u=0.
\]
The first identity says that $\operatorname{curl}_2u$ is radial; the second therefore implies $\operatorname{curl}_2u=0$. Since the disk is simply connected, $u=\nabla h$ for a smooth function $h$ satisfying $\Delta h=0$ in $D$ and $\partial_\nu h=0$ on $\partial D$. Integration by parts gives
\[
 \int_D|\nabla h|^2\,dq
 =-\int_D h\Delta h\,dq
   +\int_{\partial D}h\partial_\nu h\,d\ell=0.
\]
Thus $u=0$, and we have proved directly that
\[
 v^{(1)}(x,y)=(1-x^2-3y^2,\,2xy).
\]
As $v^{(1)}(0,0)=(1,0)$, we finally conclude that
\[
 (x_\varepsilon,y_\varepsilon)=(\varepsilon,0)+O(\varepsilon^2),
\]
which yields
$$r_\varepsilon=1+\varepsilon^2+O(\varepsilon^3)\,,\quad z_\varepsilon=O(\varepsilon^3)\,.$$
On $C$ the poloidal components vanish, but
$B=(\kappa/r_\varepsilon)\mathbf e_\phi$ does not vanish, so $C$ is a periodic orbit of $B$.
This completes the proof of the theorem.

\begin{re}
The smallness threshold $\varepsilon_0$ is independent of the regularity index and of
$m,n$. In fact, the same construction works for every constant $\kappa\in\mathbb R$, which yields MHS equilibria all whose orbits are ergodic (on each invariant toroidal surface) by taking $\kappa\not\in \mathbb Q$, or MHS equilibria all whose orbits are closed and poloidal by taking $\kappa=0$.
The periods $2\pi$ and $2\pi/n$ above refer to the normalized field
$m r^2B$. The physical periods of $B$ need not be constant across the
invariant tori.
\end{re}

\section*{Acknowledgements}

\noindent The authors are very grateful to Boris Khesin for proposing the problem on the existence of Seifert type MHS equilibria in axisymmetric toroidal domains, and explaining to us its interest in plasma physics. This work is supported by the grants CEX2023-001347-S and PID2022-136795NB-I00 (D.P.-S.) funded by MCIN/AEI /10.13039/501100011033.

\section*{Disclaimer on the use of AI tools}
\noindent Publicly available versions of ChatGPT (Sol and Astra) were used in the preparation of this manuscript. All mathematical statements and proofs have been developed and written by the authors, who take
full responsibility for the mathematical content and the final text.

\end{document}